\documentclass[12pt]{article}
\usepackage[paperwidth=8.5in, paperheight=11in, top=1in, bottom=1in, left=.5in, right=.5in]{geometry}

\usepackage{amsmath,amsfonts,amssymb,epsfig,graphicx,latexsym}
\usepackage{xspace,multicol,paracol,caption,bigstrut}
\usepackage{theorem}
\usepackage{color}

\newtheorem{theorem}{Theorem}[section]

\newtheorem{lemma}[theorem]{Lemma}
\newtheorem{question}{Question}

\numberwithin{equation}{section}

\newenvironment{proof}[1][Proof]{\begin{trivlist}
\item[\hskip \labelsep {\bfseries #1}]}{\end{trivlist}}

\newcommand{\qed}{\hfill \ensuremath{\Box}}
\newcommand{\Q}{\mathbb{Q}}
\newcommand{\R}{\mathbb{R}}
\newcommand{\Z}{\mathbb{Z}}

\begin{document}

\thispagestyle{empty}
\title{\textbf{Embedding Euclidean Distance Graphs in $\mathbb{R}^n$ and $\mathbb{Q}^n$}}

\author{\textbf{Matt Noble}\\
Department of Mathematics and Statistics\\
Middle Georgia State University\\
matthew.noble@mga.edu}
\date{}
\maketitle

\begin{abstract}

For $S \subseteq \R$, positive integer $n$, and $d > 0$, let $G(S^n, d)$ be the graph whose vertex set is $S^n$ where any two vertices are adjacent if and only if they are Euclidean distance $d$ apart.  The primary question we will consider in our work is as follows.  Given $n$ and distance $d$ actually realized as a distance between points of the rational space $\Q^n$, does there exist a finite graph $G$ that appears as a subgraph of $G(\Q^n, d)$ but not as a subgraph of $G(\R^{n-1}, 1)$?  We answer this question affirmatively for $n \leq 5$, and along the way, resolve a few related questions as well.\\
\vspace{3pt}

\end{abstract}

\section{Definitions}

Throughout, we will designate by $\R$, $\Q$, and $\Z$ the rings of real numbers, rational numbers, and integers, respectively, with each of these spaces coming equipped with the standard Euclidean distance metric.  The structure most fundamental to our work will be that of the \textit{Euclidean distance graph}.  For $S \subseteq \R$, positive integer $n$, and $d > 0$, let $G(S^n, d)$ be the graph whose vertex set is $S^n$, with any two vertices $x,y \in S^n$ being adjacent if and only if $x, y$ are distance $d$ apart.  If such a graph has vertex set $\R^n$, in all circumstances, $d$ is taken to be 1, as $G(\R^n, d)$ and $G(\R^n, 1)$ are isomorphic by an obvious scaling argument.  However, when $S$ is a proper subset of $\R$, like say, $S = \Q$, the selection of $d$ matters, as the graphs $G(\Q^n, d_1)$ and $G(\Q^n, d_2)$ may not be isomorphic when $d_1$ and $d_2$ are not rational multiples of each other.  

If $f$ is any graph parameter that can be defined on $G(S^n, d)$, we will write $f(S^n, d)$ instead of the unwieldy $f(G(S^n, d))$.  We make special note of two parameters that will play a central role in our discussion.  Let $G$ be any graph.  The \textit{clique number} $\omega(G)$ is equal to the maximum $m$ such that the complete graph $K_m$ appears as a subgraph of $G$.  Considered in our current setting, $\omega(S^n, d)$ is the largest possible number of points of $S^n$ that comprise the vertices of a regular simplex of edge-length $d$.  Originally given in \cite{eht}, the \textit{dimension} of $G$, denoted $\dim(G)$, is equal to the minimum $n$ such that $G$ appears as a subgraph of $G(\R^n, 1)$.  In the various literature on the subject (see \cite{soifer} for an expansive history), if $G$ is indeed a subgraph of $G(\R^n, 1)$, $G$ is said to be a \textit{unit-distance graph} in $\R^n$ or to have a \textit{unit-distance representation} or \textit{unit-distance embedding} in $\R^n$.  If the value $n$ is not stipulated -- that is, $G$ is just referred to as being a ``unit-distance graph" with no other specification given -- it is assumed that $\dim(G) \leq 2$.

\section{Introduction (and a Brief Interlude)}

We give in this section a brief (and perhaps amusing) digression detailing the genesis of our current work.  Those who desire to see only the main result of the article may feel free to skip ahead to the next section.

In March of 2019, the author attended the SEICCGTC held at Florida Atlantic University in Boca Raton, and between talks, chatted with Pete Johnson, discussing mathematics, life, and the current goings-on at the author's alma mater, Auburn University.  It came to knowledge that Allison Lab, a somewhat outdated building on campus, had been slated for demolition.  Pete had held the same office in the bowels of Allison for many years, and had been given a few months to relocate to a new building all the files, papers, and other sundries that one accumulates over decades of mathematical activity.  In attempting to decide what to keep and what to throw away, Pete had unearthed a number of problems, each of which at some nebulous point in the past had been quickly jotted down on a stray scrap of paper.  One of which was the following.

\begin{center} ``Let $G$ be a unit-distance graph.  Is $G$ guaranteed to be realized as a subgraph of some regular unit-distance graph $H$?"	
\end{center}

\noindent The author was able to supply an affirmative answer to this question, and it is given as Theorem \ref{regulartheorem1} below.

\begin{theorem} \label{regulartheorem1} Let $G$ be a unit-distance graph.  There exists a regular unit-distance graph $H$ such that $G$ is a subgraph of $H$.
\end{theorem}

\begin{proof} To begin, note that for any integer $r \geq 1$, there exists an $r$-regular unit-distance graph $K$.  This is easily seen by taking $K_2$ in the case of $r=1$.  For $r > 2$, one can create an $r$-regular graph $K$ by starting with an $(r-1)$-regular unit-distance graph $K'$, and then translating each point in the plane that is a vertex of $K'$ by an appropriate unit vector.  The vertex set of $K$ will then be each of those points corresponding to vertices of $K'$ along with their translates.

Let $G$ be a unit-distance graph.  Consider a unit-distance representation of $G$ in $\mathbb{R}^2$, and suppose we have maximum degree $\Delta(G) = r$.  Form a new graph $G'$ by way of the following procedure.  For each $v \in V(G)$ with $\deg(v) < r$, place $r - \deg(v)$ points on the unit circle centered at $v$ along with edges making those newly-placed vertices adjacent to $v$.  Let $P = \{v \in V(G'): \deg(v) = 1\}$. 

Suppose $K$ is an $r$-regular unit-distance graph.  Let $K' = K - uv$ for any adjacent $u,v \in V(K)$.  Select $a,b \in P$.  Place a copy of $K'$ in the plane so that $a$ and $u$ coincide.  Now place a new vertex at distance 1 from $v$.  Repeat this procedure as many times as necessary, and do so in order that the vertex $v$ in the final copy placed of $K'$ coincides with $b$.  

If $|P|$ is even, the steps outlined above can be iteratively applied to pairs of vertices $a,b \in P$ to construct an $r$-regular graph $H$ having $G$ as a subgraph.  If $|P|$ is odd, instead select $c \in P$, and then execute the same procedure for pairs of vertices in $P \setminus \{c\}$ to eventually produce a unit-distance graph $H'$ such that $\deg(c) = 1$ and all vertices in $V(H') \setminus \{c\}$ have degree $r$.  Now take $r$ copies of $H'$, and orient them such that the vertices corresponding to $c$ in each of those copies are placed at the same point in $\mathbb{R}^2$.  This can be done so that no other vertices coincide, and the resulting graph is our desired $H$.\qed
\end{proof}

For an arbitrary graph $G$, a subgraph $K$ of $G$ is defined to be \textit{induced} if for any $u,v \in V(K)$ with $uv \not \in E(K)$, then $uv \not \in E(G)$ as well.  If $G$ is an induced subgraph of $G(\R^n, 1)$, there exists a drawing of $G$ as a unit-distance graph in $\R^n$ in which it is forbidden that non-adjacent vertices be placed at a unit-distance apart.  Often in the literature, such a drawing of $G$ is referred to as a \textit{faithful embedding} (see \cite{alon} or \cite{simonovits}).  With this in mind, we may strengthen our answer to Johnson's question by showing that if the graph $G$ in the statement of the question is an induced unit-distance graph, then $H$ can be constructed as an induced unit-distance graph as well.  To do this, all we need to do is replace the graph $K$ which was used as a device in the proof of Theorem \ref{regulartheorem1} with an induced unit distance-graph $J$ having two non-adjacent $a,b \in V(J)$ satisfying $\deg(a) = \deg(b) = r-1$, and for all $v \in V(J) \setminus \{a,b\}$, $\deg(v) = r$.  We show the existence of such $J$ in the following theorem.

\begin{theorem} \label{regulartheorem2} Let $r \in \mathbb{Z}^+$ with $r \geq 2$.  There exists an induced unit-distance graph $J$ with two non-adjacent $a,b \in V(J)$ satisfying $\deg(a) = \deg(b) = r-1$, and for all $v \in V(J) \setminus \{a,b\}$, $\deg(v) = r$.
\end{theorem}

\begin{proof} We proceed by way of induction.  In the case of $r=2$, one need only let $J$ be the path $P_3$.  Now let $k \in \mathbb{Z}^+$ and suppose that $J$ is an induced unit-distance graph such that there exist non-adjacent $a,b \in V(J)$ with $\deg(a) = \deg(b) = k-1$, and for all $v \in V(J) \setminus \{a,b\}$, $\deg(v) = k$.  Position $J$ so that $a,b$ both lie on the $x$-axis.  For some small $\epsilon > 0$, define a unit vector $\phi = \langle \sqrt{1-\epsilon^2}, \epsilon \rangle$.  Let $f:\mathbb{R}^2 \rightarrow \mathbb{R}^2$ translate the plane by $\phi$.  Note that $\epsilon$ can be chosen so that the induced graph $J'$ on vertex set $V(J) \cup f(V(J))$ has vertices $a$, $b$, $f(a)$, and $f(b)$ of degree $k$ and all others of degree $k+1$.  

Let line $\mathcal{L}$ be the perpendicular bisector of $a$ and $f(b)$.  Consider a point $P$ on $\mathcal{L}$.  Let $g:\mathbb{R}^2 \rightarrow \mathbb{R}^2$ be a transformation that rotates the plane about $P$ such that $|a - g(a)| = |f(b) - g(f(b))| = 1$.  Again, note that $P$ can be selected so that the induced graph $J''$ on vertex set $V(J') \cup g(V(J'))$ has vertices $f(a)$, $b$, $g(f(a))$ and $g(b)$ of degree $k$ and all others of degree $k+1$.  Denote by $C$ the circle upon whose boundary lie $f(a)$, $b$, and $g(b)$.  Let $h:\mathbb{R}^2 \rightarrow \mathbb{R}^2$ be a transformation that rotates the plane about the center of $C$ such that $|f(a) - h(f(a))| = |b - h(b)| = |g(b) - h(g(b))| = 1$.  Let $J'''$ be the induced graph on vertex set $V(J'') \cup h(V(J''))$.  Through careful selection of $\epsilon$ and $P$ (essentially, making $\epsilon$ sufficiently small and $P$ far enough away from $a,f(b)$), we can guarantee that $J'''$ has $g(f(a))$ and $h(g(f(a)))$ being vertices of degree $k$ and all other vertices having degree $k+1$.  This completes the proof of the theorem.\qed
\end{proof}

As described above, Theorems \ref{regulartheorem1} and \ref{regulartheorem2} together give us Theorem \ref{regulartheorem3} as a corollary.

\begin{theorem} \label{regulartheorem3} Let $G$ be an induced unit-distance graph.  There exists a regular induced unit-distance graph $H$ such that $G$ is a subgraph of $H$.
\end{theorem} 

Now, the above arguments are loosely in the same mathematical neighborhood as the material of Section 3, but they do not serve as a direct lead-in.  That said, we choose to introduce our current work by way of this discussion because Johnson's question strikes us as a departure from the way problems concerning Euclidean distance graphs are usually formulated.  Typically, a problem seen in this field will begin with a list of graph characteristics and some predetermined $n$, and then ask if there exists a Euclidean distance graph $G$ in $\mathbb{R}^n$ that satisfies those properties.  A related question may then ask for the minimum possible value of some property of $G$, be it $|V(G)|$ or $|E(G)|$, or some other parameter.  For a stunning success, see de Grey's construction \cite{degrey} of a 5-chromatic unit-distance graph in the plane, and the torrent of work that followed (see \cite{exoo}, \cite{heule}, and numerous others).  Johnson's question, however, begins with the a priori knowledge that $G$ is a unit-distance graph, and then asks about the possible nature of its embedding.  This line of questioning feels (to the author, at least) as being quite fresh, and easily lending to further investigation.  In Section 3, we will pose another question along these lines, where we ask among all graphs $G$ having $\dim(G) = n$, if it is possible, for each distance $d$ realized between points of $\mathbb{Q}^n$, that some of those graphs can be drawn with their vertices being points of $\mathbb{Q}^n$ and edges of length $d$.  We ultimately answer this in the affirmative for $n \leq 5$.  In Section 4, we give a few additional thoughts on how a future attack on this question may proceed, and along the way, resolve a related problem presented in \cite{bau} concerning clique numbers of graphs $G(\mathbb{Q}^n, d)$.

\section{Embeddings in $\mathbb{Q}^n$}

The primary question considered in this section is the following.

\begin{question} \label{mainquestion} Let $n$ be a positive integer and suppose $d > 0$ is a distance realized between points of $\mathbb{Q}^n$.  Is there guaranteed to exist a finite subgraph of $G(\mathbb{Q}^n, d)$ that is not a subgraph of $G(\mathbb{R}^{n-1}, 1)$?
\end{question}

Note that an alternate formulation of Question \ref{mainquestion} would be to ask if there exists a finite subgraph $G$ of $G(\mathbb{Q}^n, d)$ with $\dim(G) = n$.  In this section we will resolve Question \ref{mainquestion} for $n \in \{1, \ldots, 5\}$ and show that in each of those cases, it has an affirmative answer.  For $n=1$, this is trivial, as by convention, the space $\R^0$ consists of a single point.  Also obvious is the case $n=2$, as the cycle $C_4$ has $\dim(C_4) = 2$ and appears as a subgraph of $G(\mathbb{Q}^2, d)$ for all $d > 0$ realized as a distance between rational points in the plane.  In considering Question \ref{mainquestion} for $n \geq 3$, we will call upon a number of past results, each of a geometric or number-theoretic sort.  

Lemma \ref{rationalsolution} appears in many introductory number theory texts, for example \cite{nagell}.

\begin{lemma} \label{rationalsolution} Consider a quadratic Diophantine equation of the form below, where $a, \ldots, f \in \mathbb{Q}$.
$$ax^2 + bxy + cy^2 + dx + ey + f = 0$$
If this equation has at least one non-trivial rational solution $(x,y)$, it has infinitely many rational solutions.
\end{lemma}

We now utilize Lemma \ref{rationalsolution} to show that the complete bipartite graph $K_{2,3}$ is a subgraph of $G(\mathbb{Q}^3, d)$ for all $d > 0$ realized as a distance between points of $\mathbb{Q}^3$.  As $\dim(K_{2,3}) = 3$ (see \cite{eht} for elaboration, or see Lemma \ref{maeharalemma} to follow), we have an affirmative answer for Question \ref{mainquestion}.

\begin{theorem} \label{nis3} The graph $K_{2,3}$ is a subgraph of $G(\mathbb{Q}^3, d)$ for every $d > 0$ realized as a distance between points of $\mathbb{Q}^3$.
\end{theorem}

\begin{proof} Suppose that $K_{2,3}$ has bipartition $\{a_1,a_2\}$ and $\{b_1,b_2,b_3\}$.  Assume $a_1$ is the origin and consider the sphere $S$ having radius $d$ and centered at the origin.  Let $b_1,b_2,b_3$ be rational points on $S$ such that the plane determined by $b_1,b_2,b_3$ does not contain the origin.  Since $b_1,b_2,b_3$ cannot be collinear, we may let $C$ be the circumcenter of the triangle whose vertices are $b_1,b_2,b_3$.  We have $C \in \mathbb{Q}^3$ by the following rationale.  Letting $P$ be the plane defined by $b_1,b_2,b_3$, note that $P$ is given by an equation of the form $n_1x + n_2y + n_3z = q$ where $n_1, n_2, n_3, q$ are each rational.  Furthermore, the line $L$ passing through the origin and $C$ is given by the parameterization $x = n_1t$, $y = n_2t$, $z = n_3t$.  The point $C$ is just the intersection of $P$ and $L$, and in substituting the above parametric equations for $x$, $y$, and $z$ into the equation for $P$, the obtained solution for $t$ is rational.  Thus $C \in \mathbb{Q}^3$.  We may then place vertex $a_2$ at the point $2C$ to complete our embedding of $K_{2,3}$.\qed
\end{proof}

In \cite{maehara}, Maehara determines the Euclidean dimension of all complete multipartite graphs.  We remark that the Euclidean dimension of a graph $G$ (often notated $\dim_E(G)$) is defined similarly to the dimension of $G$, except for the added stipulation that in any unit-distance representation of $G$ in $\mathbb{R}^n$, non-adjacent vertices of $G$ are forbidden to be placed at a unit-distance apart.  Although he does not explicitly mention it, the constructions Maehara gives in \cite{maehara} show that for a complete multipartite graph $G$, it is indeed the case that $\dim_E(G) = \dim(G)$.  We then have Lemma \ref{maeharalemma} as an immediate corollary of the main result of \cite{maehara}.   

\begin{lemma} \label{maeharalemma} Let $\alpha, \beta, \gamma$ be non-negative integers with $\alpha + \beta + \gamma \geq 2$.  Let $G$ be a complete $(\alpha + \beta + \gamma)$-partite graph having exactly $\alpha$ parts of size 1, exactly $\beta$ parts of size 2, and exactly $\gamma$ parts of size greater than or equal to 3.  If $\beta + \gamma \leq 1$, then $\dim(G) = \alpha + \beta + 2\gamma - 1$.  If $\beta + \gamma \geq 2$, $\dim(G) = \alpha + \beta + 2\gamma$.
\end{lemma}

We now show that for any $n \in \mathbb{Z}^+$, there is at least one value of $d$ which results in Question \ref{mainquestion} having an affirmative answer. 

\begin{theorem} \label{root2} Let $n \geq 2$.   Let $G$ be the complete multipartite graph having $(n-1)$ parts of size 1 and one part of size 3.  Then $G$ is a subgraph of $G(\mathbb{Q}^n, \sqrt2)$ but not a subgraph of $G(\mathbb{R}^{n-1}, 1)$.   
\end{theorem} 

\begin{proof} By Lemma \ref{maeharalemma}, $\dim(G) = n$ so $G$ is not a subgraph of $G(\mathbb{R}^{n-1}, 1)$.  To see that $G$ is indeed a subgraph of $G(\mathbb{Q}^n, \sqrt2)$, begin by labeling the partite sets of $G$ as $\{v_1\}, \ldots, \{v_{n-1}\}, \{b_1,b_2,b_3\}$.  Place $v_1, \ldots, v_{n-1} \in \mathbb{Q}^n$ where each $v_i$ has its $i^{th}$ coordinate being a 1 and all other coordinates being 0.  Now consider points of $\mathbb{R}^n$ of the form $(t, \ldots, t, w)$.  Such a point is at distance $\sqrt2$ simultaneously from each of $v_1, \ldots, v_{n-1}$ if and only if Equation \ref{eq31} holds.

\begin{equation} 
(n-2)t^2 + (t-1)^2 + w^2 = 2
\label{eq31}
\end{equation}

\noindent We know this equation has solution $t=0, w=1$.  By Lemma \ref{rationalsolution}, it has infinitely many rational solutions.  Thus $b_1,b_2,b_3$ can all be placed at points of $\mathbb{Q}^n$.\qed 
\end{proof}

We give as Lemma \ref{baulemma} the main result of the recent \cite{bau}.  Coupled with Theorem \ref{root2}, it answers Question \ref{mainquestion} affirmatively for all $n \equiv 0 \pmod 4$.

\begin{lemma} \label{baulemma} A space $\mathbb{Q}^n$ has the property that the graphs $G(\mathbb{Q}^n, d_1)$ and $G(\mathbb{Q}^n, d_2)$ are isomorphic for all pairs of distances $d_1, d_2 > 0$ realized between points of $\mathbb{Q}^n$ if and only if $n$ is equal to 1, 2, or a multiple of 4.
\end{lemma}

In \cite{beeson}, Beeson resolves the problem of deciding when a given triangle is similar to one that can be drawn with its vertices being points of $\mathbb{Z}^n$, with Lemma \ref{beesonlemma} being one of the main results.  We will state it exactly as it appears in \cite{beeson}, and then give a small amount of elaboration afterward.  

\begin{lemma} \label{beesonlemma} A triangle is embeddable in $\mathbb{Z}^4$ if and only if all of its tangents are rational multiples of $\sqrt{k}$, where $k$ is a positive integer that can be represented as a sum of three squares. 
\end{lemma}

Throughout \cite{beeson}, Beeson uses the word ``embeddable" not to mean precisely that a triangle $T$ can be drawn with its vertices being points of $\mathbb{Z}^n$, but rather that some triangle similar to $T$ can be so drawn.  This distinction ends up mattering when the setting is $\mathbb{Z}^4$, as for example, an equilateral triangle with edge-length 1 cannot be drawn with its vertices being points of $\mathbb{Z}^4$, but an equilateral triangle with side length $\sqrt2$ can.  However, for our means this is not an issue when we shift the setting to $\mathbb{Q}^4$.  Fundamental to the proof of Lemma \ref{baulemma} in \cite{bau} is the existence, for $n \equiv 0 \pmod 4$, of a bijective transformation $\varphi: \mathbb{Q}^n \rightarrow \mathbb{Q}^n$ that scales distance by a factor $\sqrt{q}$ for any $q \in \mathbb{Q}^+$.  Thus Beeson's notion of ``embeddability" can just be taken to mean that if a triangle $T$ (whose side lengths are of the form $\sqrt{a}, \sqrt{b}, \sqrt{c}$ with $a,b,c$ rational) fits the given hypotheses, then $T$ can be oriented so that its vertices are points of $\mathbb{Q}^4$.  We formalize this observation in Lemma \ref{beesonq4lemma}.  

\begin{lemma} \label{beesonq4lemma} Let $T$ be a triangle with side lengths $\sqrt{a}, \sqrt{b}, \sqrt{c}$ with $a,b,c \in \mathbb{Q}^+$.  Then $T$ can be drawn with its vertices being points of $\mathbb{Q}^4$ if and only if all of its tangents are rational multiples of $\sqrt{k}$, where $k$ is a positive integer that can be represented as a sum of three squares. 
\end{lemma}

There are a few more things to note about Lemmas \ref{beesonlemma} and \ref{beesonq4lemma}.  In the commentary in \cite{beeson}, Beeson makes an exception for $90^\circ$ angles, and says to just consider $\infty$ (that is, $\tan 90^\circ$) as also being a rational multiple of an appropriate $\sqrt{k}$.  Also, the positive integers $k$ representable as a sum of three integer squares are described by a well-known result of Gauss, where it is shown that $k = x^2 + y^2 + z^2$ for integers $x,y,z$ if and only if the square-free part of $k$ is not congruent to $7 \pmod 8$.

We now prove that Question \ref{mainquestion} has an affirmative answer for $n = 5$.  Note that the dimension of the complete tripartite graph $K_{1,3,3}$ is 5 by Lemma \ref{maeharalemma}, and also that for every $r \in \mathbb{Q}^+$ and $n \geq 4$, $\sqrt{r}$ is realized as a distance between points of $\mathbb{Q}^n$ as a consequence of Lagrange's four-square theorem.

\begin{theorem} \label{nis5} The complete tripartite graph $K_{1,3,3}$ is a subgraph of $G(\mathbb{Q}^5, \sqrt{r})$ for every $r \in \mathbb{Q}^+$.
\end{theorem}

\begin{proof} Without loss of generality, we may assume that $r$ is a positive integer.  Let $G = K_{1,3,3}$ and suppose that $G$ has partite sets $\{a_1,a_2,a_3\}$, $\{b_1,b_2,b_3\}$, and $\{c\}$.  To begin our embedding of $G$ in $\mathbb{Q}^5$, place $c$ at the origin $O = (0,0,0,0,0)$.  Let $T$ be the right triangle in the diagram below.

\begin{center}
\includegraphics[scale=.7]{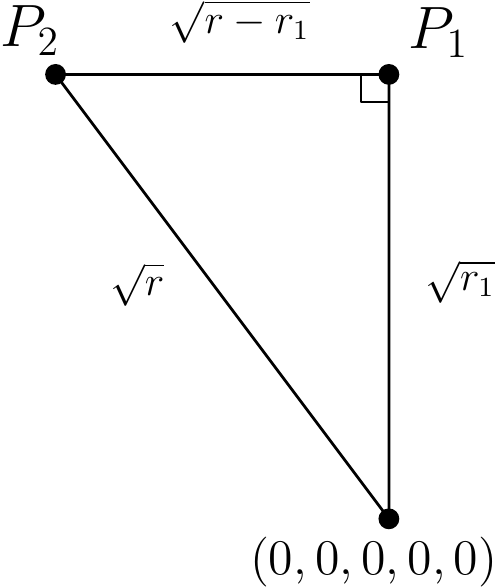}
\end{center}

Write $P_1 = (x_1,y_1,z_1,w_1,0)$ and $P_2 = (x_2,y_2,z_2,w_2,0)$ and assume $P_1, P_2 \in \mathbb{Q}^5$.  For this assumption to be valid, by Lemma \ref{beesonq4lemma} we must have $\sqrt{\frac{r-r_1}{r_1}}$ a rational multiple of $\sqrt{k}$ for some integer $k$ representable as a sum of three squares.  Note that $\sqrt{\frac{r-r_1}{r_1}}$ satisfying this property implies $\sqrt{\frac{r_1}{r - r_1}}$ does as well.

Now choose a point $P_3 \in \mathbb{Q}^5$ such $P_3$ has 0 as its fifth coordinate, and that the vector $\overrightarrow{P_1P_3}$ is orthogonal to both $\overrightarrow{OP_1}$ and $\overrightarrow{P_1P_2}$.  Let $\mathcal{P}$ be the plane defined by vectors $u = \overrightarrow{P_1P_2}$ and $v = \overrightarrow{P_1P_3}$.  Just to make sure we are clear in the terminology being used, note that $\mathcal{P}$ is a 2-dimensional plane, not a hyperplane.  Letting $u = \langle u_1,u_2,u_3,u_4,0 \rangle$ and $v = \langle v_1,v_2,v_3,v_4,0 \rangle$, we can parameterize points of $\mathcal{P}$ as
$(su_1 + tv_1 + x_1, su_2 + tv_2 + y_1, su_3 + tv_3 + z_1, su_4 + tv_4 + w_1, 0)$ for $s,t \in \mathbb{R}$.  The set of points in $\mathcal{P}$ at distance $\sqrt{r}$ from the origin are those satisfying Equation \ref{eq32}.

\begin{equation} 
(su_1 + tv_1 + x_1)^2 + (su_2 + tv_2 + y_1)^2 + (su_3 + tv_3 + z_1)^2 + (su_4 + tv_4 + w_1)^2 = r
\label{eq32}
\end{equation}

\noindent This Diophantine equation satisfies the hypotheses of Lemma \ref{rationalsolution}, with the existence of point $P_2$ implying that is has infinitely many rational solutions.  Three of these solutions can be used to place $a_1, a_2, a_3$.

Let $\mathcal{S} \subset \mathbb{R}^5$ be the set of all points simultaneously equidistant to $a_1, a_2, a_3, c$ and whose last coordinate is 0.  Set $\mathcal{S}$ consists of the intersection of three hyperplanes of $\mathbb{R}^4$, each of which has rational coefficients.  We have $\mathcal{S}$ being a line which may be parameterized as $(\alpha_1t + \beta_1, \alpha_2t + \beta_2, \alpha_3t + \beta_3, \alpha_4t + \beta_4, 0)$ where the $\alpha_i$ are rational.  We of course may use any point on the line for $\beta = (\beta_1, \beta_2, \beta_3, \beta_4, 0)$, however we will specifically set $\beta = (qx_1, qy_1, qz_1, qw_1, 0)$ where $q = \frac{r}{2r_1}$.  It is straightforward to verify that $\beta$ is equidistant from each of $a_1, a_2, a_3, c$.    

Let $\mathcal{S}' \subset \mathbb{R}^5$ be simply the set of all points simultaneously equidistant to $a_1, a_2, a_3, c$.  Any point of $\mathcal{S}'$ can be expressed as $(\alpha_1t + qx_1, \alpha_2t + qy_1, \alpha_3t + qz_1, \alpha_4t + qw_1, k)$, and for such a point to be at distance $\sqrt{r}$ from each of $a_1, a_2, a_3, c$, it must satisfy the following Diophantine equation.

\begin{equation} 
(\alpha_1t + qx_1)^2 + (\alpha_2t + qy_1)^2 + (\alpha_3t + qz_1)^2 + (\alpha_4t + qw_1)^2 + k^2 = r
\label{eq33}
\end{equation}

Setting $t=0$, we obtain $k^2 = r - q^2r_1$.  So, if $r - q^2r_1$ happens to be a non-zero rational square, we have a non-trivial rational solution for Equation \ref{eq33}, and Lemma \ref{rationalsolution} then indicates that there are infinitely many such solutions.  Just as before, we can use three of those solutions to place $b_1, b_2, b_3$ and complete our embedding of $G$ in $\mathbb{Q}^5$.  

All that needs to be done to ensure that this embedding process is successful is to select $r_1 \in \mathbb{Q}^+$ so that both of the following hold:
\begin{enumerate}
\item[$(i)$] $\sqrt{r - q^2r_1} \in \mathbb{Q}^+$.
\item[$(ii)$] $\frac{r - r_1}{r_1}$ is a sum of three rational squares.
\end{enumerate}

Our plan is as follows.  Let $r_1 = (\frac{a}{b})r$ where $a,b \in \mathbb{Z}^+$ and $a < b$.  After some simplification, we have that $(i)$ holds -- that is, $r - q^2r_1$ is a rational square -- if and only if $r(\frac{4a - b}{a})$ is a rational square.  Regarding $(ii)$, and after simplifying and then applying Gauss's condition concerning integers that are the sum of three squares, we find that $\frac{r - r_1}{r_1}$ is a sum of three rational squares if and only if $ab - a^2$ is a sum of three integer squares.  We now consider two cases.

If $r$ is even, let $a$ be some sufficiently large odd square and let $4a - b = r$.  Clearly, this selection of $a,b$ satisfies $(i)$.  Since by assumption $r$ is square-free, we have $b \equiv 2 \pmod 4$.  As any odd square is congruent to 1 modulo 4, we then have $ab - a^2 \equiv 1 \pmod 4$.  Thus $(ii)$ holds as well.

Now suppose $r$ is odd.  If $r \equiv 1 \pmod 4$, let $a = rd^2$ where $d$ is some sufficiently large odd integer.  Select $b$ so that $4a - b$ is an odd square, and note that this gives $b \equiv 3 \pmod 4$.  Clearly, $(i)$ holds.  Since $a \equiv 1 \pmod 4$, we then have $ab - a^2 \equiv 2 \pmod 4$ and $(ii)$ holds.  If $r \equiv 3 \pmod 4$, let $a = rd^2$ where $d$ is some sufficiently large even integer.  Again, select $b$ so that $4a - b$ is an odd square, which allows $(i)$ to hold and results in $b \equiv 3 \pmod 4$.  Here, $ab - a^2 = d^2(rb - r^2d^2)$, and in accordance with Gauss's condition, we have $d^2(rb - r^2d^2)$ being representable as a sum of three integer squares if and only if $rb - r^2d^2$ is representable as a sum of three integer squares.  Since $r,b$ are both congruent to 3 modulo 4 and $d$ is even, we have $ rb - r^2d^2 \equiv 1 \pmod 4$, and $(ii)$ holds.\qed

\end{proof}

\section{Concluding Thoughts}

As a reader progresses through the previous section, it would be natural to ask if the proof of Theorem \ref{root2} can be applied in a general setting.  Unfortunately, it cannot, as the proof was reliant on the existence of $K_{n}$ appearing as a subgraph of $G(\mathbb{Q}^n, \sqrt2)$.  As it turns out, there exist specific selections of $n,d$ in which the largest $m$ such that $K_m$ is a subgraph of $G(\mathbb{Q}^n, d)$ is only $m = n-3$.  In this section, we expound upon this observation, and along the way, resolve a question posed in \cite{bau}.  We then give a few thoughts on how a future resolution of Question \ref{mainquestion} may perhaps proceed.

As defined in Section 1, let $\omega(G)$ denote the clique number of $G$.  For $S^n \subseteq \mathbb{R}^n$, the \textit{clique number} $C_1(S^n)$ and \textit{lower clique number} $c_1(S^n)$ are defined in \cite{bau}, where $C_1(S^n) = \max\{\omega(S^n,d): d > 0\}$ and $c_1(S^n) = \min\{\omega(S^n,d): d > 0 \text{ and } d \text{ is actually realized as a distance between points of } S^n\}$.  The proof of the ``only if" direction of Lemma \ref{baulemma} consisted of showing that for $n \in \mathbb{Z}^+$ with $n \geq 3$ and $n \not \equiv 0 \pmod 4$, it is the case that $C_1(\mathbb{Q}^n) \neq c_1(\mathbb{Q}^n)$.  The question was then posed in \cite{bau} as to how far apart $C_1(\mathbb{Q}^n)$ and $c_1(\mathbb{Q}^n)$ can possibly be, with an example given where $C_1(\mathbb{Q}^n) - c_1(\mathbb{Q}^n) = 3$.  A well-known result of Schoenberg \cite{schoenberg} determines, for each $n$, the largest regular simplex that can be drawn with vertices in $\mathbb{Z}^n$.  Since $C_1(\mathbb{Z}^n) = C_1(\mathbb{Q}^n)$ due to a scaling argument, we will present Schoenberg's result using our current notation as Lemma \ref{restatedsch}.  

\begin{lemma} \label{restatedsch} For a positive integer $n$, $C_1(\mathbb{Q}^n) = n + 1$ in the following cases: 
\begin{enumerate}
\item[(i)] $n$ is even and $n+1$ is a perfect square.
\item[(ii)] $n \equiv 3 \pmod 4$.
\item[(iii)] $n \equiv 1 \pmod 4$ and $n+1$ is the sum of two squares.
\end{enumerate}
Otherwise, $C_1(\mathbb{Q}^n) = n$. 
\end{lemma}

With access to Lemma \ref{restatedsch}, determining the maximum possible difference $C_1(\mathbb{Q}^n) - c_1(\mathbb{Q}^n)$ is just an exercise in bounding $c_1(\mathbb{Q}^n)$ alone.  This seems particularly relevant to our current work, as any graph $G$ we hope to use in answering Question \ref{mainquestion} must have $\omega(G) \leq c_1(G)$.  

\begin{theorem} \label{bauanswer} For any positive integer $n$, $C_1(\mathbb{Q}^n) - c_1(\mathbb{Q}^n) \leq 3$.  
\end{theorem}

\begin{proof} For $n \leq 3$, we have $C_1(\mathbb{Q}^n) \leq 4$ and $c_1(\mathbb{Q}^n) \geq 2$, so we may assume $n \geq 4$.  Write $n = 4m + k$ for some $k \in \{0,1,2,3\}$.  In the case of $k = 0$, Lemma \ref{baulemma} gives $C_1(\mathbb{Q}^n) = c_1(\mathbb{Q}^n)$, and so by applying Lemma \ref{restatedsch}, we see that $C_1(\mathbb{Q}^n) - c_1(\mathbb{Q}^n) \leq 4$ for all $n$, and the only situation which could possibly result in $C_1(\mathbb{Q}^n) - c_1(\mathbb{Q}^n) = 4$ is when $\omega(\mathbb{Q}^{4m+3}, d) = 4m$ for some $d > 0$.  Without loss of generality, we may assume $d = \sqrt{r}$ for some square-free positive integer $r$.  

Let $P_1, \ldots, P_{4m} \in \mathbb{Q}^{4m}$ such that $|P_i - P_j| = \sqrt{r}$ for each $i \neq j$.  In other words, the points $P_1, \ldots, P_{4m}$ form the vertices of a copy of the complete graph $K_{4m}$ appearing as a subgraph of $G(\mathbb{Q}^{4m}, \sqrt{r})$.  For each $i \in \{1, \ldots, 4m\}$, extend $P_i$ to a point $P_i' \in \mathbb{Q}^{4m+3}$ by placing zeroes as its last three coordinate entries.  Suppose $P \in \mathbb{Q}^{4m+3}$ is a point simultaneously at a distance $\sqrt{r}$ from each of $P_1', \ldots, P_{4m}'$.  One can imagine, especially after sampling the arguments in the previous section, that the existence or non-existence of such a point $P$ can be determined through analysis of a corresponding Diophantine equation.  Indeed, that is the case, as it was shown in \cite{bau} (specifically, Theorem 3.6 of \cite{bau}) that the existence or non-existence of $P$ can be decided by determining the solvability or insolvability, respectively, in rationals, of the Diophantine equation given below.

\begin{equation} 
\frac{r(4m-1)}{8m} + 2rmx^2 + y^2 + z^2 + w^2 = r
\label{eq41}
\end{equation}

Unfortunately, the authors of \cite{bau} failed to realize that Equation \ref{eq41} is solvable for all selections of $r,m \in \mathbb{Z}^+$.  We will show this by first rewriting it as the following Equation \ref{eq42}.

\begin{equation} 
y^2 + z^2 + w^2 = \left(\frac{4rm + r}{8m}\right) - 2rmx^2 
\label{eq42}
\end{equation}

\noindent We now move to homogeneous coordinates in Equation \ref{eq43}, which has a non-trivial solution in integers if and only if Equation \ref{eq42} is solvable in rationals.

\begin{equation} 
y^2 + z^2 + w^2 = \left(\frac{4rm + r}{8m}\right)t^2 - 2rmx^2 
\label{eq43}
\end{equation}

\noindent Set $t = 4mt_0$, and note that Equation \ref{eq43} is solvable if and only if Equation \ref{eq44} is.

\begin{equation} 
y^2 + z^2 + w^2 = 2rm[(4m + 1){t_0}^2 - x^2] 
\label{eq44}
\end{equation} 

\noindent Write $2rm = s\alpha^2$ where $s$ is square-free and let $t_0 = \frac{t_1}{\alpha}$, $x = \frac{x_1}{\alpha}$, and note that Equation \ref{eq44} is solvable if and only if Equation \ref{eq45} is.

\begin{equation} 
y^2 + z^2 + w^2 = s[(4m + 1){t_1}^2 - {x_1}^2] 
\label{eq45}
\end{equation}

Denote by $\gamma$ the right-hand side of this equation.  By Gauss' characterization of integers representable as a sum of three squares, to guarantee that Equation \ref{eq45} is solvable, we just need to select integers $t_1, x_1$ so that $\gamma$ is not of the form $4^k(8a+7)$ for non-negative integers $a,k$.  This can be done as follows.  If $s \equiv 3 \pmod 4$, select $t_1 \equiv 0 \pmod 4$ and $x_1 \equiv 1 \pmod 4$.  This implies $\gamma \equiv 1 \pmod 4$.  If $s \equiv 1 \text{ or } 2 \pmod 4$, select $t_1 \equiv 1 \pmod 4$ and $x_1 \equiv 0 \pmod 4$.  This results in $\gamma \equiv 1 \pmod 4$ or $\gamma \equiv 2 \pmod 4$, respectively.  In each of these cases, $\gamma$ is representable as a sum of three squares, and we have Equation \ref{eq45} being solvable.  This concludes the proof of the theorem.\qed
\end{proof}

So, in a future attack on Question \ref{mainquestion}, one could view Theorem \ref{bauanswer} as a starting guideline for separating those graphs $G$ which merit consideration as being a subgraph of $G(\Q^n, d)$ from graphs which can immediately be set aside.  We also are led to believe that Lemma \ref{maeharalemma} may ultimately point the way for all $n > 5$.  For $n \geq 6$ with $n \equiv 2 \pmod 4$, the complete multipartite graph $G$ having $\frac{n}{2}$ parts of size 3 has $\dim(G) = n$ and is not in conflict with Theorem \ref{bauanswer}.  We leave it open as to whether $G$ is a subgraph of $G(\Q^n, d)$ for all relevant $d$.  For $n > 5$ with $n$ odd, perhaps a case could be made for the complete multipartite graph having $\frac{n-1}{2}$ parts of size 3 and one part of size 1.


\end{document}